\documentclass[12pt,reqno]{amsart}
\usepackage{amscd,amssymb,amsmath,mathabx,etex,tikz,tikz-cd,float,cancel}
\newtheorem{thm}[equation]{Theorem}
\numberwithin{equation}{section}
\newtheorem{cor}[equation]{Corollary}

\newtheorem{lem}[equation]{Lemma}

\newtheorem{defin}[equation]{Definition}

\newtheorem{prop}[equation]{Proposition}

\begin{document}
\raggedbottom \voffset=-.7truein \hoffset=0truein \vsize=8truein
\hsize=6truein \textheight=8truein \textwidth=6truein
\baselineskip=18truept

\def\mapright#1{\ \smash{\mathop{\longrightarrow}\limits^{#1}}\ }
\def\mapleft#1{\smash{\mathop{\longleftarrow}\limits^{#1}}}
\def\mapup#1{\Big\uparrow\rlap{$\vcenter {\hbox {$#1$}}$}}
\def\mapdown#1{\Big\downarrow\rlap{$\vcenter {\hbox {$\ssize{#1}$}}$}}
\def\mapne#1{\nearrow\rlap{$\vcenter {\hbox {$#1$}}$}}
\def\mapse#1{\searrow\rlap{$\vcenter {\hbox {$\ssize{#1}$}}$}}
\def\mapr#1{\smash{\mathop{\rightarrow}\limits^{#1}}}
\def\ss{\smallskip}
\def\ar{\arrow}
\def\vp{v_1^{-1}\pi}
\def\at{{\widetilde\alpha}}
\def\sm{\wedge}
\def\la{\langle}
\def\ra{\rangle}
\def\on{\operatorname}
\def\ol#1{\overline{#1}{}}
\def\spin{\on{Spin}}
\def\cat{\on{cat}}
\def\lbar{\ell}
\def\qed{\quad\rule{8pt}{8pt}\bigskip}
\def\ssize{\scriptstyle}
\def\a{\alpha}
\def\tz{tikzpicture}
\def\bz{{\Bbb Z}}
\def\Rhat{\hat{R}}
\def\im{\on{im}}
\def\ct{\widetilde{C}}
\def\ext{\on{Ext}}
\def\sq{\on{Sq}}
\def\eps{\epsilon}
\def\ar#1{\stackrel {#1}{\rightarrow}}
\def\br{{\bold R}}
\def\bC{{\bold C}}
\def\bA{{\bold A}}
\def\bB{{\bold B}}
\def\bD{{\bold D}}
\def\bh{{\bold H}}
\def\bQ{{\bold Q}}
\def\bP{{\bold P}}
\def\bx{{\bold x}}
\def\bo{{\bold{bo}}}
\def\si{\sigma}
\def\Vbar{{\overline V}}
\def\dbar{{\overline d}}
\def\wbar{{\overline w}}
\def\Sum{\sum}
\def\tfrac{\textstyle\frac}
\def\tb{\textstyle\binom}
\def\Si{\Sigma}
\def\w{\wedge}
\def\equ{\begin{equation}}
\def\AF{\operatorname{AF}}
\def\b{\beta}
\def\G{\Gamma}
\def\D{\Delta}
\def\L{\Lambda}
\def\g{\gamma}
\def\k{\kappa}
\def\psit{\widetilde{\Psi}}
\def\tht{\widetilde{\Theta}}
\def\psiu{{\underline{\Psi}}}
\def\thu{{\underline{\Theta}}}
\def\aee{A_{\text{ee}}}
\def\aeo{A_{\text{eo}}}
\def\aoo{A_{\text{oo}}}
\def\aoe{A_{\text{oe}}}
\def\vbar{{\overline v}}
\def\endeq{\end{equation}}
\def\sn{S^{2n+1}}
\def\zp{\bold Z_p}
\def\cR{{\mathcal R}}
\def\P{{\mathcal P}}
\def\cF{{\mathcal F}}
\def\cQ{{\mathcal Q}}
\def\notint{\cancel\cap}
\def\cj{{\cal J}}
\def\zt{{\bold Z}_2}
\def\bs{{\bold s}}
\def\bof{{\bold f}}
\def\bq{{\bold Q}}
\def\be{{\bold e}}
\def\Hom{\on{Hom}}
\def\ker{\on{ker}}
\def\kot{\widetilde{KO}}
\def\coker{\on{coker}}
\def\da{\downarrow}
\def\colim{\operatornamewithlimits{colim}}
\def\zphat{\bz_2^\wedge}
\def\io{\iota}
\def\Om{\Omega}
\def\Prod{\prod}
\def\e{{\cal E}}
\def\zlt{\Z_{(2)}}
\def\exp{\on{exp}}
\def\abar{{\overline a}}
\def\xbar{{\overline x}}
\def\ybar{{\overline y}}
\def\zbar{{\overline z}}
\def\Rbar{{\overline R}}
\def\nbar{{\overline n}}
\def\Gbar{{\overline G}}
\def\qbar{{\overline q}}
\def\bbar{{\overline b}}
\def\et{{\widetilde E}}
\def\ni{\noindent}
\def\coef{\on{coef}}
\def\den{\on{den}}
\def\lcm{\on{l.c.m.}}
\def\vi{v_1^{-1}}
\def\ot{\otimes}
\def\psibar{{\overline\psi}}
\def\thbar{{\overline\theta}}
\def\mhat{{\hat m}}
\def\ihat{{\hat i}}
\def\jhat{{\hat j}}
\def\khat{{\hat k}}
\def\exc{\on{exc}}
\def\ms{\medskip}
\def\ehat{{\hat e}}
\def\etao{{\eta_{\text{od}}}}
\def\etae{{\eta_{\text{ev}}}}
\def\dirlim{\operatornamewithlimits{dirlim}}
\def\Gt{\widetilde{G}}
\def\lt{\widetilde{\lambda}}
\def\st{\widetilde{s}}
\def\ft{\widetilde{f}}
\def\sgd{\on{sgd}}
\def\lfl{\lfloor}
\def\rfl{\rfloor}
\def\ord{\on{ord}}
\def\gd{{\on{gd}}}
\def\rk{{{\on{rk}}_2}}
\def\nbar{{\overline{n}}}
\def\MC{\on{MC}}
\def\lg{{\on{lg}}}
\def\cB{\mathcal{B}}
\def\cS{\mathcal{S}}
\def\cP{\mathcal{P}}
\def\N{{\Bbb N}}
\def\Z{{\Bbb Z}}
\def\Q{{\Bbb Q}}
\def\R{{\Bbb R}}
\def\C{{\Bbb C}}
\def\l{\left}
\def\r{\right}
\def\mo{\on{mod}}
\def\xt{\times}
\def\notimm{\not\subseteq}
\def\Remark{\noindent{\it  Remark}}
\def\kut{\widetilde{KU}}

\def\*#1{\mathbf{#1}}
\def\0{$\*0$}
\def\1{$\*1$}
\def\22{$(\*2,\*2)$}
\def\33{$(\*3,\*3)$}
\def\ss{\smallskip}
\def\ssum{\sum\limits}
\def\dsum{\displaystyle\sum}
\def\la{\langle}
\def\ra{\rangle}
\def\on{\operatorname}
\def\od{\text{od}}
\def\ev{\text{ev}}
\def\o{\on{o}}
\def\U{\on{U}}
\def\lg{\on{lg}}
\def\a{\alpha}
\def\bz{{\Bbb Z}}
\def\vareps{\varepsilon}
\def\bc{{\bold C}}
\def\bN{{\bold N}}
\def\nut{\widetilde{\nu}}
\def\tfrac{\textstyle\frac}
\def\b{\beta}
\def\G{\Gamma}
\def\g{\gamma}
\def\zt{{\Bbb Z}_2}
\def\pt{\widetilde{p}}
\def\zth{{\bold Z}_2^\wedge}
\def\bs{{\bold s}}
\def\bx{{\bold x}}
\def\bof{{\bold f}}
\def\bq{{\bold Q}}
\def\be{{\bold e}}
\def\lline{\rule{.6in}{.6pt}}
\def\xb{{\overline x}}
\def\xbar{{\overline x}}
\def\ybar{{\overline y}}
\def\zbar{{\overline z}}
\def\ebar{{\overline \be}}
\def\nbar{{\overline n}}
\def\zb{{\overline z}}
\def\Mbar{{\overline M}}
\def\et{{\widetilde e}}
\def\ni{\noindent}
\def\ms{\medskip}
\def\ehat{{\hat e}}
\def\what{{\widehat w}}
\def\Yhat{{\widehat Y}}
\def\Wbar{{\overline W}}
\def\minp{\min\nolimits'}
\def\mul{\on{mul}}
\def\N{{\Bbb N}}
\def\Z{{\Bbb Z}}
\def\Q{{\Bbb Q}}
\def\R{{\Bbb R}}
\def\C{{\Bbb C}}
\def\notint{\cancel\cap}
\def\se{\operatorname{secat}}
\def\cS{\mathcal S}
\def\cR{\mathcal R}
\def\el{\ell}
\def\TC{\on{TC}}
\def\dstyle{\displaystyle}
\def\ds{\dstyle}
\def\mt{\widetilde{\mu}}
\def\zcl{\on{zcl}}
\def\Vb#1{{\overline{V_{#1}}}}

\def\Remark{\noindent{\it  Remark}}
\title
{Manifold properties of planar polygon spaces}
\author{Donald M. Davis}
\address{Department of Mathematics, Lehigh University\\Bethlehem, PA 18015, USA}
\email{dmd1@lehigh.edu}
\date{May 8, 2018}


\maketitle
\begin{abstract} We prove that the tangent bundle of a generic space of planar $n$-gons with specified side lengths, identified under isometry, plus a trivial line bundle is isomorphic to $(n-2)$ times a canonical line bundle. We then discuss consequences for orientability, cobordism class,  immersions, and parallelizability.
 \end{abstract}
 \section{Main results}\label{introsec}
Let $\ell=(\ell_1,\ldots,\ell_n)$ be an $n$-tuple of positive real numbers, and let $M(\ell)$ (resp.~$\Mbar(\ell)$) denote the space of $n$-gons in the plane with successive side lengths  $\ell_1,\ldots,\ell_n$, identified under oriented isometry (resp.~isometry). These spaces have been studied by many authors. See, for example, \cite{Dcoh}, \cite{Dor}, \cite{D3}, \cite{Far}, \cite{H}, \cite{HK}, \cite{HR}, \cite{KK}, or \cite{KM}. If there is no subset $S\subset \{1,\ldots,n\}$ such that $\ds\sum_{i\in S}\ell_i=\ds\sum_{i\not\in S}\ell_i$, then $\ell$ is called {\it generic}, and $M(\ell)$ and $\Mbar(\ell)$ are $(n-3)$-manifolds. We restrict our attention to generic length vectors.

 There is a canonical double cover $p:M(\ell)\to \Mbar(\ell)$ which identifies a polygon with its reflection across a side. Associated to $p$ is a canonical line bundle $\xi$ over $\Mbar(\ell)$.
 Let $\tau(M(\ell))$ and $\tau(\Mbar(\ell))$ denote tangent bundles of these spaces. Our first theorem is
 \begin{thm}\label{tanbdlthm} There is a vector bundle isomorphism $\tau(\Mbar(\ell))\oplus\vareps\approx(n-2)\xi$, where $\vareps$ is a trivial line bundle.\end{thm}

   We were initially led to Theorem \ref{tanbdlthm} by an investigation into the Stiefel-Whitney classes of $\tau(\Mbar(\ell))$. These, of course, follow immediately from that theorem, but the discoveries came in the opposite order. We will give our original proof of the following corollary in Section \ref{SWsec}.

\begin{cor}\label{thm} Let $R=w_1(\xi)\in H^1(\Mbar(\ell);\zt)$. The total Stiefel-Whitney class of $\tau(\Mbar(\ell))$ satisfies $$w(\tau(\Mbar(\ell)))=(1+R)^{n-2}.$$\end{cor}

We deduce consequences of these results for orientability, cobordism class, immersions in Euclidean space, and parallelizability.

Our first corollary determines orientability of $\Mbar(\ell)$.
\begin{cor}\label{orcor} $\Mbar(\ell)$ is orientable iff $n$ is even or $\Mbar(\ell)$ is diffeomorphic to $(S^1)^{n-3}$.\end{cor}
\begin{proof} Since $w_1$ is the obstruction to orientability, it is immediate from Corollary \ref{thm} that $\Mbar(\ell)$ is orientable iff $n$ is even or $R=0$. But $R=0$ iff the double cover $M(\ell)\to \Mbar(\ell)$ is trivial, and this is true iff $M(\ell)$ is disconnected. It is noted in \cite[Rmk 2.8]{H} that $M(\ell)$ is disconnected iff  $\Mbar(\ell)$ is diffeomorphic to an $(n-3)$-torus.\end{proof}

Here is another corollary of Theorem \ref{thm}.
\begin{cor}\label{cor} If $R^{n-3}=0$, then $\Mbar(\ell)$ is null cobordant.  All $n$-gon spaces $\Mbar(\ell)$ which have $R^{n-3}\ne0$ are cobordant to $RP^{n-3}$.\end{cor}
\begin{proof} Let $m=n-3$. The cobordism class of an $m$-manifold is determined by which $m$-dimensional Stiefel-Whitney monomials are nonzero. By Corollary \ref{thm}, all $m$-dimensional Stiefel-Whitney monomials in $H^*(\Mbar(\ell);\zt)$ equal $R^m$, so if $R^m=0$, they are all 0. If $R^m\ne0$, then $w_{i_1}^{a_1}\cdots w_{i_k}^{a_k}$ with distinct $i_j$'s and positive $a_j$'s is nonzero iff all $\binom{m+1}{i_j}$ are odd, which is also true for $RP^m$.\end{proof}

Our third corollary involves immersions in Euclidean space.
\begin{cor}\label{cor3} If $2^e+3\le n\le 2^{e+1}$ and $R^{2^{e+1}+2-n}\ne0\in H^*(\Mbar(\ell);\zt)$, then $\Mbar(\ell)$ cannot be immersed in $\R^{2^{e+1}-2}$.\end{cor}
\begin{proof} If such an immersion exists, then the dual Stiefel-Whitney class $\overline w_{2^{e+1}+2-n}=0$, since $\Mbar(\ell)$ is an $(n-3)$-manifold. By Corollary \ref{thm}, this equals $\binom{-(n-2)}{2^{e+1}+2-n}R^{2^{e+1}+2-n}$. Since $R^{2^{e+1}+2-n}$ is assumed to be nonzero and $\binom{-(n-2)}{2^{e+1}+2-n}\equiv\binom{2^{e+1}-1}{n-3}\ne0\in\zt$, we obtain a contradiction to the assumed immersion.\end{proof}

Note that if $n=2^e+3$, this nonimmersion would be optimal, since the $2^e$-manifold $\Mbar(\ell)$ certainly immerses in $\R^{2^{e+1}-1}$.

Perhaps our most interesting result regards the parallelizability of $\Mbar(\ell)$. The proof of this appears in Section \ref{pfsec}.
\begin{thm}\label{pthm} Let $\ell=(\ell_1,\ldots,\ell_n)$ be a generic length vector.
\begin{itemize}
\item[a.] If $n$ is odd, then $\Mbar(\ell)$ is parallelizable iff it is diffeomorphic to the $(n-3)$-torus $T^{n-3}$.
\item[b.] If $n=6$ or $10$, $\Mbar(\ell)$ is parallelizable.
\item[c.] Let $n\equiv0\pmod4$ with $n\ge8$. Then $\Mbar(\ell)$ is parallelizable if it is diffeomorphic to $T^{n-3}$ or the $n$-dimensional Klein bottle of \cite{Kn}. If $\ell=(0^{n-5},1,1,1,2,2)$, then the parallelizability of $\Mbar(\ell)$ is not known. Otherwise, $\Mbar(\ell)$ is not parallelizable.
\end{itemize}
\end{thm}

For $n\le13$, Theorem \ref{pthm} determines the parallelizability of all spaces $\Mbar(\ell)$ except two, one with $n=8$ and one with $n=12$.
The 0-lengths in Theorem \ref{pthm}(c) are small sides such that if there are $k$ 0's (denoted $0^k$), the length of each is less than $1/k$.

We also remark here briefly about the classification of the spaces $\Mbar(\ell)$ (\cite{HR}).
These spaces are classified completely, up to diffeomorphism, by their {\it genetic code}, which is a set of subsets, called {\it genes}, of $[n]:=\{1,\ldots,n\}$. We will define these in Section \ref{pfsec}. The genetic codes are listed for $n\le6$ in \cite{HR} and for $n\le 9$ in \cite{web}. For $6\le n\le9$, the number of diffeomorphism classes of nonempty $n$-gon spaces $\Mbar(\ell)$ is given in Table \ref{T1}.

\begin{table}[H]
\caption{Number of nonempty $n$-gon spaces $\Mbar(\ell)$}
\label{T1}
\begin{tabular}{c|cccc}
$n$&$6$&$7$&$8$&$9$\\
\hline
&$20$&$134$&$2469$&$175427$
\end{tabular}
\end{table}

In Section \ref{Rsec}, we discuss how to tell, in terms of the genetic code of $\Mbar(\ell)$, whether certain powers of $R$ are nonzero. In particular, we determine for each of the 134 7-gon spaces their cobordism class and whether Corollary \ref{cor3} can be used to obtain an optimal nonimmersion of these 4-manifolds in $\R^6$.

\section{Proofs}\label{pfsec}
In this section, we prove Theorems \ref{tanbdlthm} and \ref{pthm}.

En route to proving Theorem \ref{tanbdlthm}, we will also note the following result, which was pointed out to us by Jean-Claude Hausmann. We thank him for his help on various matters.
 \begin{thm} \label{Mbdl} The vector bundle $\tau(M(\ell))\oplus\vareps$ is isomorphic to a trivial bundle.\end{thm}
 \begin{proof}[Proof of Theorems \ref{Mbdl} and \ref{tanbdlthm}] The manifold $M(\ell)$ can be defined as $F^{-1}(\ell_n)$, where $F:(S^1)^{n-1}\to\C$ is defined by
 $$F(z_1,\ldots,z_{n-1})=\sum_{j=1}^{n-1}\ell_jz_j.$$
 See, e.g., \cite[(1.2)]{Far}.
 Since $\ell$ is generic, $\ell_n$ is a regular value of $F$, i.e., $F^{-1}(\ell_n)\notint\{\pm1\}^{n-1}$, (see e.g., \cite[Thm 3.1]{Hf}), and, moreover, there is an
 $\epsilon$-neighborhood $U$ of $\ell_n$ such that $F^{-1}(U)\notint\{\pm1\}^{n-1}$. The set $W=F^{-1}(U)$ is an open subset of $(S^1)^{n-1}-\{\pm1\}^{n-1}$ and is acted on freely by the involution $\phi$ defined, using complex conjugation, by $\phi(z_1,\ldots,z_{n-1})=(\zb_1,\ldots,\zb_{n-1})$. The manifold $W$ is parallelizable with vector fields $v_1,\ldots,v_{n-1}$ defined by
 $$v_j(z_1,\ldots,z_{n-1})=(0,\ldots,0,iz_j,0,\ldots,0).$$
 Let $\iota:M(\ell)\to W$ denote the inclusion map. Note that the tubular neighborhood $W$ can be considered to be the normal bundle $\nu$ of $M(\ell)$, which is trivial by the construction of $M(\ell)$ and $W$ using $F$. The trivial bundle $\iota^*(\tau(W))$ is isomorphic to $\tau(M(\ell))\oplus\nu$. We obtain
 $$(n-1)\vareps\approx\tau(M(\ell))\oplus2\vareps.$$
 This implies Theorem \ref{Mbdl} by Lemma \ref{stprop}.

 Let $\Wbar$ be the quotient of $W$ by the free involution $\phi$, and note that $\Mbar(\ell)$ is the quotient of $M(\ell)$ by $\phi$. The double cover $p$ defined above is the restriction to $M(\ell)$ of the double cover $W\to \Wbar$. Let $\xi_W$ denote the associated line bundle. Then $(n-1)\xi_W$ can be given by $W\times\R^{n-1}/(w,\la t_j\ra)\sim(\phi(w),\la-t_j\ra)$, and there is a vector bundle isomorphism $\tau(\Wbar)\approx (n-1)\xi_W$ defined by
 $$(w,\sum t_jv_j(w))\leftrightarrow (w,\la t_j\ra).$$
 This is well-defined since $\phi_*:\tau(W)\to \tau(W)$ satisfies $\phi_*(v_j(w))=-v_j(\phi(w))$.

 The normal bundle $\overline\nu$ of $\Mbar(\ell)$ is isomorphic to $M(\ell)\times\R\times\R/(x,s,t)\sim(\phi(x),s,-t)$, which is isomorphic to $\vareps\oplus\xi$. We obtain
 \begin{equation}\label{bdliso}(n-1)\xi\approx\overline\iota^*(\tau(\Wbar))\approx\tau(\Mbar(\ell))\oplus\overline\nu\approx \tau(\Mbar(\ell))\oplus\vareps\oplus\xi.\end{equation}
 There exists a vector bundle $\theta$ over $\Mbar(\ell)$ such that $\xi\oplus\theta$ is isomorphic to a trivial bundle. Adding $\theta$ to both sides of (\ref{bdliso}), we obtain that $\tau(\Mbar(\ell))\oplus\vareps$ is stably isomorphic to $(n-2)\xi$. Theorem \ref{tanbdlthm} now follows from Lemma \ref{stprop}.\end{proof}

The following lemma, which is certainly well-known to experts, was used above.
\begin{lem} Let $\theta$ and $\eta$ be stably isomorphic $(m+1)$-plane bundles over an $m$-dimensional CW-complex $X$. Assume also that if $m+1$ is even, then $w_1(\theta)=0=w_1(\eta)$. Then $\theta$ and $\eta$ are isomorphic.\label{stprop}\end{lem}
\begin{proof} Let $B_SO(m+1)$ (resp.~$B_SO$) be $BSO(m+1)$ (resp.~$BSO$) if $m+1$ is even, and $BO(m+1)$ (resp.~$BO$) if $m+1$ is odd, and let $G=\Z$ if $m+1$ is even, and $\zt$ if $m+1$ is odd.
  Let $f$ and $g$ be the maps $X\to B_SO(m+1)$ classifying $\theta$ and $\eta$, and $i:B_SO(m+1)\to B_SO$ the usual inclusion. The hypothesis is that $i\circ f\simeq i\circ g$. As a fibration, $i$ is orientable in the stable range. (e.g., \cite[Cor 5.2(iii)]{Nuss}) Thus $i$ has a Moore-Postnikov tower which, through dimension $m+1$,  is a fiber sequence
  $$K(G,m+1)\to B_SO(m+1)\to B_SO \to K(G,m+2).$$
 There is an action map $\mu:K(G,m+1)\times B_SO(m+1)\to B_SO(m+1)$, and $g=\mu(c\times f)$ for some $c:X\to K(G,m+1)$. Since $X$ is $m$-dimensional, $c$ is trivial, and hence $g\simeq f$.
  \end{proof}

\begin{proof}[Proof of Theorem \ref{pthm}] Part (a) is immediate from the proof of Corollary \ref{orcor}, which notes that if $n$ is odd and not diffeomorphic to $T^{n-3}$, then $w_1(\tau(\Mbar(\ell)))\ne0$, so $\Mbar(\ell)$ is not parallelizable. Part (b) when $n=6$ follows from Corollary \ref{orcor} and the well-known result (\cite{St}) that every compact orientable 3-manifold is parallelizable.

Now we prove part (b) when $n=10$. In this case $\Mbar(\ell)$ and $M(\ell)$ are 7-manifolds, and so the Atiyah-Hirzebruch spectral sequence gives a commutative diagram of short exact sequences

$$\begin{CD} 0@>>>H^4(\Mbar(\ell);\Z)@>\qbar>>\kot(\Mbar(\ell))@>>>F_2(\Mbar(\ell))@>>>0\\
@. @VVp_1^*V @VVp_2^*V@VVp_3^*V@.\\
0@>>>H^4(M(\ell);\Z)@>q>>\kot(M(\ell))@>>>F_2(M(\ell))@>>>0,\end{CD}$$
in which $F_2(\ )$ is an extension of $H^1(\ ;\zt)$ and $H^2(\ ;\zt)$, and hence is a group of order 2 or 4.

Since $p:M(\ell)\to \Mbar(\ell)$ is a double cover, $2\ker(p_1^*)=0$. The pullback $p^*(\xi)$ is a trivial bundle, and hence $p_2^*([\xi])=0$. [\![This pullback bundle is
$$\{(x,x,t),(x,\phi(x),t)\in M(\ell)\times M(\ell)\times\R\}/(x,x,t)\sim(x,\phi(x),-t),$$
which maps to $\Mbar(\ell)\times\R$ by sending $(x,x,t)$ and $(x,\phi(x),-t)$ to $([x],t)$.]\!]

There exists $\a\in H^4(\Mbar(\ell))$ such that $\qbar(\a)=[4\xi]$. Since $p_2^*([4\xi])=0$, the diagram implies that $\a\in\ker(p_1^*)$, and hence $2\a=0$. Thus $[8\xi]=0$, and hence $\tau(\Mbar(\ell))$ is stably trivial by Theorem \ref{tanbdlthm}. Thus $\Mbar(\ell)$ is stably parallelizable, and hence is parallelizable by \cite{BK}, which shows that a stably parallelizable 7-manifold is parallelizable.

Part (c) follows from Corollary \ref{thm} and Theorem \ref{R2thm} together with the observation that (a) if $\Mbar(\ell)$ has genetic code $\{n,n-3,\ldots,1\}$, then $\Mbar(\ell)$ is diffeomorphic to $T^{n-3}$, which is parallelizable, (b) if $\Mbar(\ell)$ has genetic code $\{n,n-4,\ldots,1\}$, then it is diffeomorphic to the $(n-3)$-dimensional Klein bottle of \cite{Kn}, which was shown there to be parallelizable if and only if $n-3$ is odd, and (c) if $\ell=(0^{n-5},1,1,1,2,2)$, then the genetic code of $\Mbar(\ell)$ is $\{n,n-2,n-5,\ldots,1\}$.\end{proof}

\begin{thm}\label{R2thm} If $n\ge5$, then $R^2=0$ in $H^*(\Mbar(\ell);\zt)$ iff the genetic code of $\Mbar(\ell)$ is $\{n,n-3,\ldots,1\}$, $\{n,n-4,\ldots,1\}$, or $\{n,n-2,n-5,\ldots,1\}$.\end{thm}

In order to prove Theorem \ref{R2thm}, we recall, in Proposition \ref{gprop}, our interpretation (\cite[Thm 2.1]{D3}) of Hausmann-Knudson's determination (\cite{HK}) of the algebra $H^*(\Mbar(\ell);\zt)$.
All sets in a genetic code contain the integer $n$, and so we say that a {\it gee} is a gene with the $n$ omitted, and a geetic code is the genetic code without listing $n$ in each of the sets. There is a partial order on the set of genes or gees by $S=\{s_1,\ldots,s_k\}\le T$ if $T$ contains a subset $\{t_1,\ldots,t_k\}$ with $s_i\le t_i$ for all $i$. A {\it subgee} is any set of positive integers which is $\le$ some gee, and gees are the maximal subgees.

  \begin{prop}\label{gprop} $H^*(\Mbar;\zt)$ is generated by 1-dimensional classes $R$ and $V_1,\ldots,V_{n-1}$ with only relations as below, where $V_S:=\ds\prod_{i\in S}V_i$,
  \begin{itemize}
  \item $V_S$ is zero unless $S$ is a subgee;
  \item $V_i^2=RV_i$;
  \item If $S$ is a subgee with $|S|\ge n-2-d$, then
  $$\sum_{T\notint S}R^{d-|T|}V_T=0.$$
  \end{itemize}
  \end{prop}

  \ni We denote the third relation here by $\cR_S$.

For the rest of this section, we are dealing with $n$-gon spaces.
\begin{defin} For $G\subseteq[n-1]$, let $\Gt=[n-1]-G$ and $\Gbar=\Gt-\max\{i:\ i\in\Gt\}$.\end{defin}
\begin{lem} If $G_1$ and $G_2$ are subsets of $[n-1]$ (possibly equal) with $G_2\ge \Gbar_1$, then $G_1$ and $G_2$ cannot both be subgees of the same genetic code.\label{Glem}\end{lem}
\begin{proof} This is the only place that we need to know the relationship between genes and length vectors. A subset $S$ of $[n]$ is defined to be {\it short} if $\ds\sum_{i\in S}\ell_i<\sum_{i\not\in S}\ell_i$. Then a set $T\subset[n-1]$ is a subgee iff $T\cup\{n\}$ is short. If $G_1$ is a subgee, then $\Gt_1$ is not short (called {\it long}). Therefore $\{n\}\cup\Gbar_1$ is long, and hence so is $\{n\}\cup G_2$. Thus $G_2$ is not a subgee.\end{proof}

\begin{proof}[Proof of Theorem \ref{R2thm}] We introduce some notational shortcuts: $k,,1$ for $k,\ldots,1$ and $k,,\ihat,,1$ for $k,\ldots,i+1,i-1,\ldots,1$. Basically a double comma means $,\ldots,$; i.e., include all intermediate numbers.
Also, all cohomology groups have coefficients in $\zt$.

By Proposition \ref{gprop}, the relations in $H^2(\Mbar(\ell))$ are associated to subgees of size $\ge n-4$. In the next paragraph, we will show that the only possible subgees of size $\ge n-4$ are (a) $\{n-3,,1\}$; (b) $\{n-4,,1\}$;
(c) $\{n-3,,\ihat,,1\}$ with $\{1\le i<n-3\}$; (d) $\{n-2,n-5,,1\}$; and (e) $\{n-1,n-5,,1\}$.

We show that $G\ge \Gbar$ in all other cases, and so $G$ is not a subgee by Lemma \ref{Glem}. If $G=[n-1]-\{i\}$, then $\Gbar=\emptyset$, so $G\ge\Gbar$. If $G=[n-1]-\{i,j\}$, $i>j$, then $\Gbar=\{j\}$, and so $G\ge\Gbar$ unless $(i,j)=(n-1,n-2)$. If $G=[n-1]-\{i,j,k\}$, $i>j>k$, then  $\Gbar=\{j,k\}$, and so $G\ge \Gbar$ unless $(i,j)=(n-1,n-2)$ or $(j,k)=(n-3,n-4)$.

The group $H^2(\Mbar(\ell))$ is spanned by $R^2$ and all $V_iV_j$, where $1\le i\le j\le k_0$, where $k_0$ is the largest integer contained in any of the gees of $\ell$. If the geetic code of $\ell$ does not have any gees of length $\ge n-4$, then there are no relations among these classes, and so $R^2\ne0$.
We now consider geetic codes having a gee of type (a) through (e) above, plus perhaps other gees.

{\bf Case (a)}: If $G=\{n-3,,1\}$ appears alone in the geetic code, then $H^2(\Mbar(\ell))$ is spanned by $R^2$ and $V_iV_j$, $1\le i\le j\le n-3$. Since none of these sets $\{i,j\}$ is disjoint from $G$, the relation $\cR_G$ is exactly $R^2=0$. This $G$ (as $G_1$) cannot be accompanied in a geetic code by Lemma \ref{Glem}, because an accompanying $G_2$ cannot satisfy $G_2\ge\Gbar=\{n-2\}$, but any such $G_2$ is $\le G_1$, and hence cannot appear separately in the geetic code, since geetic codes only include maximal subgees.

The remaining cases deal with the situation when the largest gee has size $n-4$. Note that the possible subgees of size $n-4$ are totally ordered by $\ge$, and so the set of subgees of size $n-4$ will be exactly those which are $\le $ the single gee of size $n-4$.

{\bf Case (b)}: If $G=\{n-4,,1\}$ appears alone in the geetic code, then the relation $R^2=0$ is obtained from $\cR_G$ as in Case (a). This $G$ can be accompanied in the geetic code by gees $G'$ containing an integer $i>n-4$, but such $G'$ must have length $<n-4$, else it would be $>G$, contradicting maximality of $G$. Thus $G'$ does not add a new relation, but now $\cR_G$ says $0=R^2+V_i^2$ (plus possibly other $V_j^2$). Hence $R^2\ne0$.

{\bf Case (c)}: If $G=\{n-3,,\ihat,,1\}$ appears alone in the geetic code, it will have relations $\cR_{n-3,,\jhat,,1}$ for all $j\le i$, which is $R^2+V_j^2=0$, since there are no 2-subsets of $[n-3]$ disjoint from $G$. Clearly, no combination of these relations can yield $R^2=0$. This $G$ can be accompanied in the geetic code, but, as noted above, not by $G'$ of size $\ge n-4$. Thus there are no additional relations. The accompanying $G'$'s may add additional basis elements to $H^2(\Mbar(\ell))$, such as $V_{n-2}^2$, but these will not affect the impossibility, already noted, of obtaining $R^2=0$ as a consequence of the relations.

{\bf Case (d)}: If $G=\{n-2,n-5,,1\}$ appears alone in the geetic code, it will have relations $\cR_G$, $\cR_{n-3,n-5,,1}$, and $\cR_{n-4,,1}$. Then $H^2(\Mbar(\ell))$ is spanned by classes $R^2$, $V_1^2,\ldots,V_{n-2}^2$, and $V_iV_j$ with $i>j$ and $j\le n-5$. No $V_iV_j$ appears in any of the relations. The three relations are $R^2+V_{n-3}^2+V_{n-4}^2$, $R^2+V_{n-2}^2+V_{n-4}^2$, and $R^2+V_{n-2}^2+V_{n-3}^2$. Adding yields $R^2=0$. This $G$ can be accompanied in the geetic code, but an accompanying $G'$ cannot be $\ge\Gbar=\{n-3,n-4\}$ by the lemma, and so its second largest element must by $\le n-5$. It must contain $n-1$, else it would be $<G$. Then the three relations all contain the term $V_{n-1}^2$, and so their sum is no longer just $R^2$.

{\bf Case (e)}: If $G=\{n-1,n-5,,1\}$ appears alone in the geetic code, there are four relations, each of the form $R^2+T$, where $T$ is the sum of any three of $\{V_{n-1}^2,V_{n-2}^2,V_{n-3}^2,V_{n-4}^2\}$. No combination of these can equal $R^2$. Any $G'$ which would accompany $G$ in the geetic code cannot be $\ge\Gbar=\{n-3,n-4\}$, so its second largest element must be $\le n-5$, and so it is $\le G$. Thus this $G$ cannot be accompanied in the geetic code.
\end{proof}

\section{Specific results for 7-gon spaces}\label{Rsec}
The genetic codes of the 134 7-gon spaces $\Mbar(\ell)$ are listed in \cite{web}. These are connected 4-manifolds, and we can use {\tt Maple} to determine for each whether $R^4\ne0$ (which is, by Corollary \ref{cor}, equivalent to it being cobordant to $RP^4$) and whether $R^3\ne0$ (which is equivalent to having Corollary \ref{cor3} imply that it does not immerse in $\R^6$). We first state the results, and then describe the algorithm.
\begin{prop} Of the $134$ $7$-gon spaces, $72$ are cobordant to $\emptyset$, and $62$ are cobordant to $RP^4$.\end{prop}
\begin{prop} Of the $134$ $7$-gon spaces, $122$ have $R^3\ne0$ and hence cannot be immersed in $\R^6$. The ones with $R^3=0$ are those with geetic codes
$$21,\ 41,\ 61,\ 65,\ 321,\ 421,\ 521,\ 621, \ 4321,\ \{321,51\},\ \{421,61\},\ \{431,51\}.$$
\end{prop}
\ni Here we concatenate, and omit $\{-\}$ from monogenic codes; e.g., $421$ means $\{4,2,1\}$.

By Proposition \ref{gprop}, a presentation matrix for $H^{n-3}(\Mbar(\ell);\zt)$ has columns (generators) for all subgees $S$ (including $\emptyset$, which corresponds to $R^{n-3}$) and rows (relations) for all subgees $T$ except $\emptyset$. An entry is 1 iff $S$ and $T$ are disjoint, else 0. We know that $\dim(H^{n-3}(\Mbar(\ell);\zt))=1$, and so, if this presentation matrix is $(r-1)$-by-$r$, then its rank is $r-1$. Then $R=0$ iff a row-reduced form of the matrix has a row with its only 1 being in the $R$-column, and this is true iff, when the $R$-column is omitted, the matrix has rank $r-2$. So we just form the matrix without the $R$-column and ask {\tt Maple} whether its rank (over $\zt$) is less than its number of columns.

Similarly, $H^{n-4}(\Mbar(\ell);\zt)$ has the same columns, but now rows for all subgees of size $\ge2$, filled in according to the same prescription. We know that $$\dim(H^{n-4}(\Mbar(\ell);\zt))=\dim(H^1(\Mbar(\ell);\zt)),$$ and this equals the number of subgees of size $\le1$, and so $R^{n-4}=0$ iff, when the $R$-column is removed, the rank of the resulting matrix is one less than its number of rows.

In \cite{Dcoh}, we determined a formula for $R^{n-3}$ in a monogenic genetic code, for arbitrary $n$. As noted in Corollary \ref{cor}, the mod-2 value of $R^{n-3}$ determines
whether $\Mbar(\ell)$ is cobordant to $\emptyset$ or to $RP^{n-3}$.
\begin{prop}\label{Dprop} If the genetic code of $\Mbar(\ell)$ is $\{n,g_1,\ldots,g_k\}$, let $a_i=g_i-g_{i+1}>0$ $(a_k=g_k)$. Then $$R^{n-3}=\sum_B\prod_{i=1}^k\tbinom{a_i+b_i-2}{b_i}\in\zt,$$
where $B=(b_1,\ldots,b_k)$ ranges over all $k$-tuples of nonnegative integers satisfying
$b_1+\cdots+b_\ell\le\ell$ for $1\le\ell\le k$ with equality if $\ell=k$.\end{prop}

One can tell from the genetic code whether or not  $\Mbar(\ell)$ has a nonzero vector field.
\begin{prop} \label{vfprop}Let $d_i$ denote the number of subgees of size $i$. Then an $n$-gon space $\Mbar(\ell)$ has a nonzero vector field iff $n$ is even or $\ds\sum_{i\ge0}(-1)^id_i=0$.\end{prop}
\begin{proof} We use the well-known result of Hopf that a connected manifold has a nonzero vector field iff its Euler characteristic is 0. If $n$ is even, the result follows since the Euler characteristic of an odd-dimensional manifold is 0. That the alternating sum of $d_i$'s gives the Euler characteristic of $\Mbar(\ell)$ appears as a remark at the end of Section 4 of \cite{HK}. We prove it by noting that (e.g., \cite[Thm 1.7]{Far} or \cite[Thm 2.3]{Dor}) the Betti numbers of the double cover $M(\ell)$ are given by counting subgees and their dual classes.
Thus, if $\dim(M(\ell))$ is even, $\chi(M(\ell))=2\sum(-1)^id_i$, and $\chi(\Mbar(\ell))=\frac12\chi(M(\ell))$.\end{proof}

We list the geetic codes  of the 30 cases with $n=7$ that have Euler characteristic 0. This is obtained from \cite{web}.
\begin{eqnarray*}&&1,21,31,41,51,61,321,421,521,621,431,4321,\{321,41\},\{321,51\},\{321,61\}\\
&&\{421,51\},\{421,61\},\{431,51\},\{431,61\},\{521,61\},\{32,4\},\{42,6\},\{32,41,5\},\{32,51,6\}\\
&&\{321,42,5\},\{321,43,6\},\{421,43,5\},\{421,52,6\},\{321,42,51,6\},\{421,43,51,6\}
\end{eqnarray*}
For example, $\{421,51\}$ has $d_0=1$, $d_1=5$, $d_2=6$ (21, 31, 41, 51, 32, 42), and $d_3=2$ (321, 421).
\section{Original proof of Corollary \ref{thm}}\label{SWsec}
As noted in the introduction, we obtained Corollary \ref{thm} prior to Theorem \ref{tanbdlthm}. In this section, we give that original proof.
  Throughout this section, we let $m=n-3=\dim(\Mbar(\ell))$
We use the following well-known relationship between the Stiefel-Whitney classes of the tangent bundle and the Wu classes. (e.g., \cite{MS})
\begin{prop}\label{Wu} Let $M$ be an $m$-manifold. The Wu class $v_i\in H^i(M;\zt)$ is defined to be the unique class which satisfies $v_i\cup x=\sq^i(x)$ for all $x\in H^{m-i}(M;\zt)$. Then
the total Stiefel-Whitney class, $w(\tau(M))$, of the tangent bundle of $M$ equals $\sq(v)$, where $\sq$ is the total Steenrod square and $v=\sum_{i=0}^{[m/2]}v_i$ is the total Wu class.\end{prop}

The following key lemma gives a surprisingly simple formula for  the Wu classes of $\Mbar(\ell)$.
\begin{lem}\label{Wulem} $v_i=\binom{m-i}i R^i$.\end{lem}
\begin{proof} For this result, all we need to know about $H^*(\Mbar(\ell);\zt)$ is that it is generated as an algebra by 1-dimensional classes $R,V_1,\ldots,V_{n-1}$ with relations $V_a^2=RV_a$.(\cite{HK})  There are additional relations, but we don't need them here. In general, for a product of 1-dimensional classes $x_j$,
$$\sq^i(x_1\cdots x_k)=x_1\cdots x_k\cdot\sum_{|S|=i}\prod_{j\in S}x_j,$$
where $S$ ranges over all $i$-subsets of $\{1,\ldots,k\}$. Using the relations $V_a^2=RV_a$, $H^{m-i}(\Mbar(\ell);\zt)$ is spanned by classes $R^{m-i-j}V_{a_1}\cdots V_{a_j}$, and $$\sq^i(R^{m-i-j}V_{a_1}\cdots V_{a_j})=\tbinom{m-i}iR^{m-j}V_{a_1}\cdots V_{a_j}.$$ These classes may be zero, depending on the other, more complicated relations, but still it is the case that $\sq^i$ acts as multiplication by $\binom{m-i}iR^i$.\end{proof}

Now we can prove Theorem \ref{thm}, using a combinatorial result proved below.
\begin{proof}[Proof of Corollary \ref{thm}] The first part of (\ref{eqs}) follows from Proposition \ref{Wu} and Lemma \ref{Wulem}, while the next-to-last $=$ is Corollary \ref{combcor}.
\begin{equation}\label{eqs}w(\tau(\Mbar(\ell)))=\sum_{j\ge0}\sq^j\sum_{i\ge0}\tbinom{m-i}iR^i=\sum_{k\le m}R^k\sum_i\tbinom{m-i}i\tbinom i{k-i}=\sum_{k\le m}\tbinom{m+1}kR^k=(1+R)^{m+1},\end{equation}
since $R^{m+1}=0$.
\end{proof}

In the remainder of this section, we prove
the mod-2 combinatorial result, Corollary \ref{combcor}, which was used in the above proof. This result and the integral combinatorial results, Lemma \ref{comblem}, Corollary \ref{combcor2}, and Theorem \ref{combthm}, which we use to derive it, are probably known, but we could not find them. Nor could we find a proof simpler than the rather elaborate proof that we present here.
We use the usual convention that $\binom mk=m(m-1)\cdots (m-k+1)/k!$ for any integer $m$ and nonnegative integer $k$.

\begin{lem}\label{comblem} If $m$ is an integer, and $k$ a  nonnegative integer, then
$$\sum_{i=0}^k\tbinom{m-i}i\tbinom{i-m+k}{k-i}=\sum_{i=0}^{k+2}\tbinom{m-i}i\tbinom{i-m+k+2}{k+2-i}.$$
\end{lem}
\begin{proof} We use the {\tt Maple} program {\tt Zeil}, as described in \cite[ch.6]{AB}.\footnote{It is called {\tt ct} in \cite{AB}, but runs as {\tt Zeil} in our implementation.}
It discovers that if $f(k,i)=\binom{m-i}i\tbinom{i-m+k}{k-i}$ and
$$G(k,i)=\tbinom{m-i}i\tbinom{i-m+k}{k-i}\tfrac{(2k-m+3)(-m+i-1)i(2i-m)}{(k+1-i)(k+2-i)}$$ for $i\le k$,
then
\begin{equation}\label{zeq}(k+2)(k-m+1)(f(k,i)-f(k+2,i))=G(k,i+1)-G(k,i)\end{equation}
for $i\le k-1$.
(We verified this directly in many cases, but for a complete proof, we rely on the software.) Applying $\sum_{i=0}^{k-1}$ to (\ref{zeq}), we obtain
$$(k+2)(k-m+1)(\Delta-S)=G(k,k)-G(k,0),$$
where $\Delta$ is the difference (LHS minus RHS) of the two sums in our lemma, and
$$S=\tbinom{m-k}k(1-\tbinom{2k-m+2}2)-\tbinom{m-k-1}{k+1}\tbinom{2k-m+3}1-\tbinom{m-k-2}{k+2}.$$ We note that $G(k,0)=0$. We will show
\begin{equation}\label{star}-(k+2)(k-m+1)S=G(k,k),\end{equation}
which implies that $\Delta=0$, except perhaps if $k-m+1=0$. If $k-m+1=0$, then both sides of the lemma are easily seen to equal 1 if $k$ is odd, and 0 if $k$ is even.

To prove (\ref{star}), we factor out $\binom{m-k}k$, and then (\ref{star}) becomes
\begin{eqnarray*}&&(k+2)(m-k-1)\bigl(1-\tbinom{2k-m+2}2+\tfrac{(m-2k+1)(m-2k)(m-2k-3)}{(m-k)(k+1)}\\
&&-\tfrac{(m-2k)(m-2k-1)(m-2k-2)(m-2k-3)}{(k+2)(k+1)(m-k)(m-k-1)}\bigr)\\
&=&(2k-m+3)(-m+k-1)k(2k-m)/2,\end{eqnarray*}
which was verified symbolically by {\tt Maple}.
\end{proof}
\begin{cor}\label{combcor2}  If $m$ is an integer and $k$ a nonnegative integer, then
$$\sum_{i=0}^k\tbinom{m-i}i\tbinom{i-m+k}{k-i}=\begin{cases}1&k\text{ even}\\ 0&k\text{ odd.}\end{cases}$$
\end{cor}
\begin{proof} We easily verify when $k=0$ and 1, and then apply Lemma \ref{comblem}.\end{proof}
\begin{thm}\label{combthm} If $d\ge k-m$, then
$$\sum_{i=0}^k\tbinom{m-i}i\tbinom{i+d}{k-i}=\sum_{j=0}^{[k/2]}\tbinom{m+d-1-2j}{k-2j}.$$
\end{thm}
\begin{proof} The proof is by induction on $m+d-k$, and when this is fixed, induction on $k$. The theorem is valid when $m+d-k=0$ by Corollary \ref{combcor2} since $$\sum\tbinom{k-1-2j}{k-2j}=\begin{cases}1&k\text{ even}\\0&k\text{ odd.}\end{cases}$$ It is also valid when $k=0$, since both equal 1. Assume the result for smaller values. Then using Pascal's formula at the beginning and end, and the induction hypothesis in the middle, we have
\begin{eqnarray*}&&\sum\tbinom{m-i}i\tbinom{i+d}{k-i}=\sum\tbinom{m-i}i\biggl(\tbinom{i+d-1}{k-i}+\tbinom{i+d-1}{k-1-i}\biggr)\\
&=&\sum\biggr(\tbinom{m+d-2-2j}{k-2j}+\tbinom{m+d-2-2j}{k-1-2j}\biggr)=\sum\tbinom{m+d-1-2j}{k-2j}.\end{eqnarray*}
\end{proof}
\begin{cor}\label{combcor} If $m\ge k$, then
$$\sum_{i=0}^k\tbinom{m-i}i\tbinom i{k-i}\equiv\tbinom{m+1}k\pmod2.$$
\end{cor}
\begin{proof} By Theorem \ref{combthm}, the LHS equals $\ds\sum_j\tbinom{m-1-2j}{k-2j}$. This equals 1 if $k=0$ and is $\equiv m+1\ (2)$ if $k=m$. Both $\ds\sum_j\tbinom{m-1-2j}{k-2j}$ and the RHS satisfy Pascal's formula, and they agree when $k=0$ or $m$. Hence they are equal.\end{proof}

 \def\line{\rule{.6in}{.6pt}}

\end{document}